\documentclass{amsart}

\usepackage{amsmath}
\usepackage{amssymb}
\usepackage{graphicx}
\usepackage[all]{xy}

\theoremstyle{plain}

\newtheorem{lemma}{Lemma}[section]
\newtheorem{proposition}[lemma]{Proposition}
\newtheorem{theorem}[lemma]{Theorem}

\newtheorem{cor}[lemma]{Corollary}
\newtheorem*{prop_A}{Proposition A}

\theoremstyle{definition}

\newtheorem{remark}[lemma]{Remark}

\newtheorem{claim}[lemma]{Claim}
\title[Jones representations]{Remarks on   \\ the faithfulness  
of the Jones representations}

\author[Y. Kasahara]{Yasushi Kasahara}

\address{
Department of Mathematics \\
  Kochi University of Technology \\ Tosayamada,  Kami City, Kochi \\ 
  782-8502 Japan 
}
\email{kasahara.yasushi@kochi-tech.ac.jp}

\dedicatory{Dedicated to Professor Shigeyuki Morita \\ on his
60th birthday}

\begin{document}
%
%
\newcommand{\Q}{\mathbb Q}
\newcommand{\C}{\mathbb C}
\newcommand{\D}{\mathcal{D}}
\newcommand{\g}{\gamma}
\newcommand{\G}{\Gamma}
\newcommand{\gr}{\mathcal{G}}
\newcommand{\Gd}{\gr \delta}
\newcommand{\HE}{\mathcal H}
\newcommand{\K}{\mathcal K}
\newcommand{\h}{\mathfrak h}
\newcommand{\I}{{\mathcal I}}
\renewcommand{\S}{\Sigma}
\newcommand{\sgn}{\operatorname{sgn}}
\newcommand{\rank}{\operatorname{rank}}
\newcommand{\Sp}{\operatorname{Sp}}
\newcommand{\PSp}{\operatorname{PSp}}
\newcommand{\GL}{\operatorname{GL}}
\newcommand{\ctr}{\operatorname{Center}}
\newcommand{\Int}{\operatorname{Int}}
\newcommand{\Homeo}{\operatorname{Homeo}}
\renewcommand{\L}{\mathcal L}
\newcommand{\M}{{\mathcal M}}
\newcommand{\N}{{\mathcal N}}
\newcommand{\p}{\psi}
\newcommand{\x}{\xi}
\newcommand{\z}{\zeta}
\newcommand{\Z}{{\mathbb Z}}
\newcommand{\Aut}{\operatorname{Aut}}
\newcommand{\End}{\operatorname{End}}
\renewcommand{\hom}{\operatorname{Hom}}
\renewcommand{\ker}{\operatorname{Ker}}
\newcommand{\im}{\operatorname{Im}}
\newcommand{\var}[1]{\varphi^{(#1)}}
\newcommand{\IW}{{H(q, n)}}
\newcommand{\PSL}{\operatorname{PSL}}
\newcommand{\F}{\Z/p\Z}
%
%
\begin{abstract}
We consider the linear representations of the mapping class group of an 
$n$--punctured $2$--sphere constructed by V. F. R. Jones using Iwahori--Hecke 
algebras of type A. We show that their faithfulness is equivalent to that of 
certain related Iwahori--Hecke algebra representation of Artin's braid group 
of $n-1$ strands. In the case of  $n=6$, we provide a further restriction for the 
kernel using our previous result, as well as a certain relation to the Burau 
representation of degree $4$. 
\end{abstract}

\maketitle

\section{Introduction}
\par 

A linear representation of a group is said to be faithful if it is injective as 
a homomorphism of the group into the corresponding group of  linear 
transformations. 
In the seminal paper \cite{jones},  V. F. R. Jones  constructed a family of 
linear representations of $\M_0^n$, the mapping class group of an $n$--punctured 
$2$--sphere with  one parameter. Each of the family is obtained as a modification of the 
Iwahori--Hecke algebra representation of $B_n$, Artin's  braid group of $n$ strands, provided with a {\em rectangular} Young diagram with $n$ boxes. 
Except for the trivial cases, which correspond to the Young diagrams $[1^n]$ or $[n]$, 
it  remains open whether these Jones representations of $\M_0^n$ 
are faithful or not. 
It might be therefore possible that some of these Jones representations gives a 
{\em naturally defined faithful} linear representation of $\M_0^n$, which seems 
missing even after the work of 
Korkmaz \cite{korkmaz}, and Bigelow--Budney \cite{bigelow-budney}, who 
independently constructed a faithful linear representation of $\M_0^n$ as an 
{\em induced} representation of  a faithful representation of a certain subgroup of finite 
index defined by modifying 
the Lawrence--Krammer representation of $B_{n-1}$, the faithfulness of which 
was established by the celebrated works of  Bigelow \cite{bigelow} and Krammer \cite{krammer}.
On the other hand, if there exists a nontrivial element  in the kernel of a Jones representation of $\M_0^n$, probably of infinite order,  it might be expected to give  a nontrivial knot 
with the same value of the Jones polynomial as the unknot. 
\par

In this note, we show, for arbitrary $n \geq 6$,  that the faithfulness of the Jones 
representation of $\M_0^n$ corresponding to the rectangular Young diagram $Y$ is equivalent to that of the 
Iwahori--Hecke algebra representation of  the braid group $B_{n-1}$ which corresponds 
to the {\em unique} Young diagram obtained from $Y$ by removing a single box. 
We note that the faithfulness problem for the latter representation also remains open.
Furthermore, in the case of $n=6$, we apply our previous result 
\cite{kasahara:1} to obtain a  restriction for the kernel of the Jones representation of $\M_0^6$
in terms of the mapping class group of genus $2$  via the Birman--Hilden theory. 
The case $n < 6$ is mentioned in Remark \ref{exception} (1), and its details 
are discussed elsewhere \cite{kasahara:3}.
\par

\section{Preliminaries}
 \par
 We fix some notation and briefly recall necessary material to describe Jones' construction.
 \par
 \subsection{Mapping class groups of genus $0$}
 \par
Let  $D^2$ be a $2$--disk, $P$ the set of distinct $n$ points $p_1$, $p_2$, \ldots, $p_n$ in 
 $\Int D^2$. We call $P$ the set of ``punctures". We denote by $D_n$  the pair of spaces $(D^2, P)$.
 The $n$--strand braid group $B_n$ is defined as $\pi_0 \Homeo^+ ( D_n; \partial D_n )$, the mapping class group of $D_n$. Namely, it is  the group of 
the orientation preserving homeomorphisms of $D_n$ which restrict to the identity on the boundary, 
modulo the isotopy in the same class of homeomorphisms. 
\par

We choose a simple closed curve in $\Int D^2$ so that the disk component of its complement
intersects with $P$ at $P_0 = \{ p_1, \ldots, p_{n-1} \}$. Identifying the bounding $2$--disk with $D^2$ 
again, the inclusion defines $D_{n-1} \hookrightarrow D_n$, which induces an injective homomorphism 
$$ i: B_{n-1} \to B_n $$
by extending the homeomorphisms with the identity on $D_n \smallsetminus D_{n-1}$.
\par

Next, capping off a $2$--disk on $\partial D^2$, we obtain a $2$--sphere $S^2$. 
We choose a single point $p_{n+1}$ in the interior of the added disk, and set $P_+ = P \cup \{ p_{n+1} \}$.
As in the case of $D^2$,  we set $S_n = ( S^2, P )$ and $S_{n+1} = ( S^2, P_+ )$, and define their 
mapping class groups as $\M_0^n = \pi_0 \Homeo^+ ( S_n )$ and 
$\M_0^{n+1} = \pi_0 \Homeo^+ ( S_{n+1} )$, respectively. 
We denote by $\M_0^{n+1} ( p_{n+1} )$ the subgroup of $\M_0^{n+1}$ consisting of those mapping classes 
which fix $p_{n+1}$. By ``forgetting" $p_{n+1}$, we obtain the natural surjective homomorphism 
$$ p: \M_0^{n+1} ( p_{n+1} ) \to \M_0^n $$
The inclusion $D_n \hookrightarrow S_{n+1}$ induces the homomorphism $j_n : B_n \to \M_0^{n+1}$ by 
again extending the homeomorphism with the identity. 
\begin{proposition} \label{prop:exact}
	For $n \geq 2$, there exists the following short exact sequence:
		\begin{equation}
	 0\; \xrightarrow \;  \Z\; \xrightarrow \; B_n\; \xrightarrow{j_n} 
	 \; \M_0^{n+1}(p_{n+1})\; \xrightarrow \; 1
           \label{exact}
      \end{equation}
      Here, the image of $\mathbb{Z}$ in $B_{n}$ is generated by the single element 
      $D_{\partial}$, the Dehn twist about  the simple closed curve parallel to the 
      boundary, and coincides with $\ctr(B_{n})$ if  $n \geq 3$.  If $n = 2$,   
      $\ctr(B_{2})$ coincides with $B_{2}$ and the image of $\mathbb{Z}$ is 
      an index $2$ subgroup of $\ctr(B_{2})$.
\end{proposition}
\begin{remark}
	In the case of $n \geq 3$, this proposition is nothing but 
	\cite[Lemma 2.2]{bigelow-budney}.  In the case of $n = 2$, the proposition 
	follows from the fact that $\M_{0}^{3}$ is isomorphic to the symmetric 
	group of  three letters.
\end{remark}
\par

Finally, we denote by $k: B_{n} \to \M_{0}^{n}$ the homomomorphism induced by 
the inclusion $D_{n} \hookrightarrow S_{n}$, and obtain the commutative diagram:
$$\xymatrix{
		&	{B_n}  \ar@{->>}[d]_j \ar[dr]^k	\\
	{B_n / \langle D_\partial \rangle} \ar@{=}[r]	& {\M_0^{n+1}(p_{n+1})} \ar[r]_-{p}  & {\M_0^n}
	}$$
where $\langle D_\partial \rangle$ denotes the subgroup generated by 
$D_\partial$.
\subsection{Iwahori--Hecke algebra representations of $B_n$}
\par

We denote by $H(q, n)$ the Iwahori--Hecke algebra of type $A_{n-1}$.
As its ground ring, we take $\Q(q)$, the quotient field of the polynomial ring $\Q[q]$
with $q$ an indeterminate. Let $\sigma$ be the right-handed half  twist about 
an arbitrary embedded arc in $\text{Int} D^2$ joining two distinct points of $P$
such that no interior points of the arc intersect with $P$. Such  $\sigma$ is 
unique up to conjugation in $B_n$. Then $H(q, n)$ can be defined as the quotient 
of the group ring $\Q(q) [B_n]$ of $B_n$ by the two--sided ideal generated by the single 
element $(\sigma - 1)( \sigma + q)$. Hence the projection defines a natural 
multiplication--preserving mapping $B_n \to H(q, n)$. From any representation of 
$H(q, n)$, this mapping gives rise to a representation of $B_n$.
\par

Here, we summarize the facts necessary in this note about $H(q, n )$.
For details, we refer to \cite{mathas}.
\begin{proposition} \label{fact-1}
As a $\Q(q)$--algebra, $H(q, n)$ is semisimple. Furthermore, 
all the irreducible representations of $H(q, n)$ are in one-to-one correspondence 
with all the  Young diagrams with $n$ boxes.
\end{proposition}
Let $Y$ be a  Young diagram with $n$ boxes. 
We denote by $V_Y$ the representation space of the corresponding representation
of $H(q, n)$, and by $\pi_Y : B_n \to \GL(V_Y)$ the  representation obtained
 from $V_Y$ as above. The identity element of $\GL(V_Y)$ will be denoted by  $I$.
 As usual,  we adopt the notation that $\pi_{[n]}$ denotes the one--dimensional scalar representation $\sigma \mapsto q \cdot I$, which coincides with the notation of \cite{jones}.
Then the representation $\pi_{[1^n]}$ is the one--dimensional representation
defined by $\sigma \mapsto -I$. We call this representation $\sgn$.

\begin{proposition} \label{fact-2}
Let $Y$ be an arbitrary Young diagram with $n$ boxes. We denote  all the distinct 
Young diagrams obtained from $Y$ by removing a single box 
by $Y_0^1$, $Y_0^2$, \ldots, $Y_0^s$. Then the representation of $B_{n-1}$ obtained as the composition of $\pi_Y$ with the injection $i: B_{n-1} \to B_n$ is 
equivalent to the direct sum 
$\pi_{Y_0^1} \oplus \pi_{Y_0^2} \oplus \cdots \oplus \pi_{Y_0^s}$ as 
representations over $\Q(q)$.
\end{proposition}
Next, let $\mathfrak{S}_n$ denote the permutation group of $n$ letters, and 
$\nu : B_n \to \mathfrak{S}_n$ the homomorphism induced by the permutation of 
the puncture set $P$. 
\begin{proposition} \label{fact-3}
Let $Y$ be a  Young diagram with $n$ boxes. Then the specialization of the  representation $\pi_Y$ of $B_n$ at $q=1$  descends via $\nu$ to the irreducible representation of 
$\mathfrak{S}_n$ over $\Q$  which corresponds to  the same Young 
diagram $Y$.
\end{proposition}
\par

\subsection{Jones' construction.} \label{construction}
\par

Let $Y$ be an arbitrary Young diagram with 
$n$ boxes, and $\pi_Y$ the corresponding irreducible representation of $B_n$ with the represnentation space $V_Y$. We denote by  $d$ ($=d_Y$) the dimension of 
$V_Y$ over $\Q(q)$. According to the analysis in \cite{jones}, the Dehn twist 
$D_\partial$ along the boundary curve is mapped under $\pi_Y$ to the scalar 
$q^{rn(n-1)/d}$.  Here, $r$ is a certain non-negative integer defined as 
$\rank{( I + \pi_{Y}(\sigma) )}$. A  precise combinatorial 
description of $r$ can be found in \cite{jones}. Important here is the fact that $r$ is 
equal to $0$ if and only if $Y = [ 1^{n} ]$. We also remark that $rn(n-1)/d$ is always
an integer.
\par

We now adjust  $\pi_{Y}$ by rescaling so that the image of $D_{\partial}$ 
becomes trivial,  appealing to the well-known fact that the abelianization 
of $B_{n}$ is $\Z$. For that purpose, we need the formal power $q^{1/d}$.
To avoid confusion, we introduce another indeterminate corresponding to 
$q^{1/d}$.  Let $t$ be another indeterminate. We consider the $\Q$--algebra 
homomorphism $\Q(q) \to \Q(t)$ defined by $q \mapsto t^d$. This homomorphism 
naturally gives rise to a structure of a $\Q(q)$--algebra on $\Q(t)$. From now on, 
we consider every representation over $\Q(q)$ also as that over $\Q(t)$ by coefficient extension.
\par

We denote the representation obtained as the composition of   the abelianization 
of $B_{n}$ with the mapping 
$m \in \Z \mapsto ( t^{-r} )^{m}$ by
$$ \alpha :  B_{n} \to \GL( \Q( t ) ).$$
Clearly, $\alpha$ is trivial if and only if $r = 0$, {\em i.e.}, if and only if
$Y = [ 1^{n} ]$.
We now consider the representation $\alpha \otimes_{\Q(q)} \pi_Y$.
Its representation space is  $M_Y = \Q(t) \otimes_{\Q(q)} V_Y$
which is a $d$-dimensional  vector space over $\Q(t)$. 
Since the abelianization maps $D_{\partial}$ to $n ( n - 1 ) \in \Z$, we have
$\alpha \otimes_{\Q(q)} \pi_Y ( D_{\partial} ) = I$. Therefore, by Proposition 
\ref{prop:exact}, the representation $\alpha \otimes_{\Q(q)} \pi_Y$ descends 
 via $j_{n}$ to that of $\M_0^{n+1}(p_{n+1})$. The condition that this 
 representation further descends to that of  $\M_0^n$ is given in \cite{jones}:
\begin{proposition} 
\label{sphere-jones}
Via the homomorphism $k: B_{n} \to \M_0^n$, 
the representation $\alpha \otimes_{\Q(q)} \pi_Y$ descends to that of $\M_0^n$ 
if and only if $Y$ is {\em rectangular}.
\end{proposition}
For a rectangular Young diagram $Y$ with $n$ boxes, we call the representation of 
$\M_0^n$ given by Proposition \ref{sphere-jones} the Jones representation of the 
$n$--punctured sphere corresponding to $Y$, and denote it by 
$\bar{\pi}_Y: \M_0^n \to \GL(M_Y)$, with $M_Y = V_Y \otimes \Q(t)$.
\par

\section{The faithfulness of $\bar{\pi}_{Y}$}
\par

With the preparation above, we can now state our main result.
\begin{theorem} \label{thm-A}
Suppose that $Y$ is a rectangular Young diagram with $n$ boxes, 
$\neq [ n ]$, $[ 1^{n} ]$.
Let $Y_{0}$ be the {\em unique}
Young diagram obtained from $Y$ by removing a 
single box. Assume further that $n \geq 6$.  Then the representation 
$$\bar{\pi}_{Y}: \M_{0}^{n} \to \GL (M_{Y})$$
is faithful if and only if  the representation
$$\pi_{Y_{0}} : B_{n-1} \to \GL (V_{Y_{0}})$$
is faithful.
\end{theorem}
\begin{remark} \label{exception}
(1)  For $n < 6$, there exists only a single case where the other 
assumptions of Theorem  \ref{thm-A} on $Y$ hold: $n=4$,   
$Y=[ 2, 2 ]$, and $Y_{0} = [ 2, 1 ]$. In this case,  the theorem is not true. 
In fact, on the one hand, $\pi_{Y_{0}}$ is the reduced Burau representation of 
$B_{3}$, and is classically known to be faithful \cite{magnus-peluso}. 
On the other hand, we can verify
that the kernel of $\bar{\pi}_{Y}$ is isomorphic to $\Z/2\Z \oplus \Z/2\Z$, and 
coincides with the kernel of the natural homomorphism $\M_{0}^{4} \to \PSL (2, \Z)$ described in \cite{AMU}.  In another word, $\bar{\pi}_{[ 2, 2 ]}$ descends to a faithful representation of $\PSL ( 2, \Z )$.  We also note $\PSL(2, \Z) \cong B_3 /\ctr(B_3)$. We discuss the details in \cite{kasahara:3}.
\par

(2) If $n$ is prime,  there exist no Young diagrams satisfying the assumptions 
of the theorem.  Therefore, we cannot expect that the Jones representations here would 
provide  faithful representations of $B_{n}$ for {\em all $n$}, different from the representations 
obtained from the Lawrence--Krammer representations.
\end{remark}
\begin{proof}[Proof of Theorem \ref{thm-A}.]
Via the injective homomorphism $i : B_{n-1} \to B_n$, we consider $B_{n-1}$ as a subgroup
of $B_n$. Since $Y_0$ is the unique Young diagram obtained from $Y$ by removing a 
single box, the restriction of $\pi_Y$ to $B_{n-1}$ coincides with $\pi_{Y_0}$ by Proposition 
\ref{fact-2}.  Hence we have the following commutative diagram:
$$
\xymatrix{
		&	 {B_{n-1}} \ar[ddl]_-{j_{n-1}} \ar[d]_-{i} \ar@/^/[ddrr]^{\alpha_0 \otimes \pi_{Y_0}} \\
		&	{B_n} \ar[d]_-{k} \ar[drr]^{ \alpha \otimes \pi_{Y} } \\
		{\M_0^n(p_n)} \ar[r] &	{\M_0^n} \ar[rr]_-{\bar{\pi}_Y} &  & {\GL(M_Y)} 
	}$$
where $\alpha_0$ denotes the restriction of $\alpha$ to $B_{n-1}$.
\par

Suppose now that $\bar{\pi}_Y$ is faithful. For $x \in B_{n-1}$, $\bar{x}$ denotes the 
corresponding element $j_{n-1} ( x )$ in $\M_0^n ( p_n )$. If $\pi_{Y_0} ( x ) = 1$, then 
$$\bar{\pi}_Y ( \bar{x} ) = \pi_{Y_0} ( x ) \cdot \alpha_0 ( x ) = \alpha ( x ) \cdot I. $$
Therefore, by the faithfulness of $\bar{\pi}_Y$, $\bar{x}$ lies in the center of $\M_0^n$.
On the other hand, by Gillette--Buskirk \cite{gillette-buskirk}, it is known that 
the center of $\M_0^n$ is trivial. Hence, we have $\bar{x} = 1$, that is, $x \in \ker{j_{n-1}}$. 
Now, by Proposition  \ref{prop:exact} for $n-1$, $x \in \ctr(B_{n-1})$. Then the faithfulness of 
$\pi_{Y_0}$ reduces to its faithfulness  on $\ctr(B_{n-1})$, which can be seen as follows.
Since $Y \neq [ 1^n ]$, we can see that $\alpha$, and therefore $\alpha_0$,  is non-trivial, 
as noted in Section \ref{construction}. Since $\alpha_0$ factors through the abelianization of 
$B_{n-1}$,  it is faithful on $\ctr(B_{n-1})$.  Together with the fact that 
$\alpha_0 \otimes \pi_{Y_0}$ is trivial on $\ker{j_{n-1}} = \ctr(B_{n-1})$,  
this implies that $\pi_{Y_0}$ is faithful on $\ctr(B_{n-1})$. This completes the proof 
that $\pi_{Y_0}$ is faithful.
\par

Suppose next that $\pi_{Y_0}$ is faithful. Assume that $f  \in \M_0^n$ and $\bar{\pi}_Y ( f ) = 1$.
By Proposition \ref{fact-3},  the specialization of $\bar{\pi}_Y$ at $t = 1$ descends to 
the irreducible representation of the symmetric group $\mathfrak{S}_n$ corresponding to $Y$.
By the assumption $Y \neq [ 1^n ]$, $[ n ]$, we can see that this specialization is a faithful 
representation of $\mathfrak{S}_n$. In fact, since $n \geq 5$,  it is a classical fact  that the alternating group $\mathfrak{A}_n$ is a simple finite group.  Therefore, the kernel of the specialization is  either  
$\mathfrak{S}_n$, $\mathfrak{A}_n$, or $\{ 1 \}$.  But since $Y \neq [ 1^n ]$, $[ n ]$, we can 
conclude that the kernel is not $\mathfrak{S}_n$ nor $\mathfrak{A}_n$. Hence the specialization 
is faithful on $\mathfrak{S}_n$. Therefore,  the permutation of  the set of punctures $P$  
induced by $f$ is trivial. In particular, we have $f \in \M_0^n( p_n )$. Hence, there exists some 
$x \in B_{n-1}$ such that $\bar{x} = f$. Then we have
$$\alpha ( x )  \cdot  \pi_{Y_0}( x )  = \bar{\pi}_Y( f )  = I.$$
So, we have $\pi_{Y_0}( x ) = \alpha ( x ) ^{-1} \cdot I$, and then by the assumption 
that $\pi_{Y_0}$ is faithful, $x \in \ctr( B_{n-1} )$. Hence, 
$f = \bar{x} = 1 \in \M_0^n$. Therefore, $\ker{\bar{\pi}_Y} = \{ 1 \}$. This 
completes the proof of Theorem \ref{thm-A}.
\end{proof}
\begin{remark}
In the course of the proof above, the use of the result in \cite{gillette-buskirk},
to show $x \in \ctr(B_{n-1})$ under the asumpstion that $\bar{\pi}_Y$ be faithful and 
$\pi_{Y_0}(x) = 1$, is not necessary as follows. 
Since $\bar{x} \in \ctr(\M_0^n)$, it holds  for any 
$\tau \in B_{n-1}$ that   $[ \bar{x}, \bar{\tau} ] = 1$ in $\M_0^n$. Hence 
$[ x, \tau ] \in \ker{j_{n-1}} = \ctr(B_{n-1})$. On the other hand, we see 
$\pi_{Y_0}[ x, \tau ] = [ \pi_{Y_0}( x ), \pi_{Y_0}( \tau ) ]  = [ 1, \pi_{Y_0}( \tau ) ] =1$.
Since we have already seen the faithfulness of $\pi_{Y_0}$ on $\ctr( B_{n-1} )$, 
we have $[ x, \tau ] = 1$ for any $\tau \in B_{n-1}$, {\em i.e.}, $x \in \ctr( B_{n-1} )$.
\end{remark}

\par

\section{Hyperelliptic mapping class group}
\par
In order to consider the case $n=6$ further, we need to consider the 
Jones representation of $\M_0^n$,  with $n$ even, as a representation of 
hyperelliptic mapping class group, as follows. Let $\Sigma_g$ be a closed oriented surface of genus $g \geq 2$ and let $\M_g$ be its
mapping class group. The hyperelliptic mapping class group $\HE_g$ of genus $g$
is defined as the subgroup of $\M_g$ consisting of  those elements which commute with 
the class of a fixed hyperelliptic involution $\iota: \Sigma_g \to \Sigma_g$. We identify 
the quotient orbifold  of $\Sigma_g$ by the action of $\iota$ with $S_{2g+2}$ where 
the singular locus of the orbifold corresponds to P, the set of punctures. Due to 
Birman--Hilden \cite{birman-hilden},  the natural projection $\Sigma_g \to S_{2g+2}$
induces the following short exact sequence:
\begin{equation*}
	 0\; \xrightarrow \;  \Z/2\Z\; \xrightarrow \; \HE_g\; \xrightarrow{q} 
	 \; \M_0^{2g+2}\; \xrightarrow \; 1
\end{equation*}
where the image of $\Z/2\Z$ in $\HE_g$ is generated by the class  of $\iota$. 
Provided with 
an arbitrarily {\em rectangular} Young diagram $Y$ with $( 2g + 2 )$ 
boxes, we obtain,  following Jones \cite{jones},  the representation of $\HE_g$
$$ \rho_Y: \HE_g \to \GL (M_Y)$$
as the composition of  $\bar{\pi}_Y$ with the above homomorphism $q$. 
As a corollary to Theorem \ref{thm-A}, we have:
\begin{cor} \label{cor-A}
$\ker{\rho_Y} \cong \Z/2\Z$ if and only if $\pi_{Y_0}$ is faithful.
\end{cor}
\par

\section{The case of $g = 2$}
\par

In the case of  $g = 2$,  it is classically known that $\HE_2 = \M_2$. Also, 
because of the ``column-row symmetry",  the representation of $\M_2$ obtained as 
above is essentially unique (c.f.  \cite[Section 2.2]{kasahara:1}).  So, we take
$Y = [ 2^3 ]$, and consider the representation
$$ \rho = \rho_{[ 2^3 ]}: \M_2 \to \GL(M_{[ 2^3 ]}).$$
We can now say a little bit  about  $\ker{\rho}$, and hence about 
$\ker{\bar{\pi}_{[2^3]}}$.
Let  
$$ \rho_0 : \M_2 \to \Sp( 4, \Z )$$
be the symplectic representation induced by the action of $\M_2$ on the 
homology group $H_1 ( \Sigma_2; \Z)$.
The Torelli group $\I_2$ is defined as $\ker{\rho_0}$.
\begin{theorem} \label{thm-B}
$\ker{\rho} = \Z/2\Z \oplus ( \ker{\rho} \cap \I_2 )$. Here, $\Z/2\Z$ is generated by 
the class of the hyperelliptic involution $\iota$. In particular, $\rho$ is {\em faithful on 
$\I_2$}  if and only if $\pi_{[ 2^2, 1 ]}$ is a faithful representation of $B_5$.
\end{theorem}
By the definition  $\rho = \bar{\pi}_{[ 2^3 ]} \circ q$,  we have
$\ker{\bar{\pi}_{[ 2^3 ]}} = q ( \ker{\rho} )$. Therefore, Theorem \ref{thm-B} 
implies immediately a  restriction of the kernel 
for the case $n=6$: 
\begin{cor} \label{cor-B}
	$\ker{\bar{\pi}_{[ 2^3 ]} } \subset q( \I_2 ) \cong \I_2$.
\end{cor}
\par
\begin{proof}[Proof of Theorem \ref{thm-B}] It suffices to prove the isomorphism in 
the first part of the theorem since the rest of the theorem is then a direct consequence of 
Corollary \ref{cor-A}.  We  construct  the  isomorphism directly. 
Set
$K = \ker{\rho}$, and $K_0 = K \cap \I_2$. Note that 
$\iota$ lies in $K$ by the definition of $\rho$. We then define a mapping 
$ h_0: \Z \oplus K_0 \to K$
by $h_0( a, x ) = \iota^a \cdot x$. Since $\iota^2=1$, $h_0$ descends to 
a well-defined mapping 
$$ h: \Z/2\Z \oplus K_0 \to K.$$
 Hereafter, the multiplication in $\Z/2\Z$ is written
additively, and each element of $\Z/2\Z$ will be denoted by its representative in $\Z$, so that $1+1=0$ in $\Z/2\Z$. Since $\iota$ lies in $\ctr{(\M_2)}$, the 
mapping $h$ is actually a homomorphism of group. 
Suppose next that
$h( a, x ) = 1$, {\em i.e.}, $\iota^a \cdot x =1$. Then, by taking the image under
$\rho_0$,  we have $( -I )^a \cdot \rho_0(x) = I$. Since $\rho_0(x) = I$,
we have $( -I )^a = 1$. Therefore, we have $a = 0$ in $\Z/2\Z$, and hence
$x=1$. This shows that the homomorphism $h$ is injective.
\par
Next, to prove that $h$ is surjective, 
we consider the specialization of $\rho$ at $t = -1$.
By our previous result \cite{kasahara:1},  this specialization is trivial on $\I_2$ 
and can be described as follows. Let $\sgn$ denote the one-dimensional 
representation of $\M_2$ which sends the Dehn twist along every non-separating
simple closed curve to $-1$. It is easy to see that $\sgn$ is trivial on $\I_2$, 
and therefore descends to a representation of $\Sp(4, \Z)$, denoted by 
$\overline{\sgn}$. Now, let $\lambda$ denote the linear representation of 
$\Sp(4, \Z)$ which is induced by the natural action on 
$\Lambda^2 H_1(\Sigma_2; \Z)/\omega\Z$ where $\omega$ denotes the 
symplectic class in $\Lambda^2 H_1 ( \Sigma_2 ;  \Z )$. Then the 
specialization of $\rho$ at $t=-1$ is equivalent to 
$(\overline{\sgn} \otimes \lambda) \circ \rho_0$ (\cite[Lemma 2.1]{kasahara:1}).
\begin{claim} \label{claim}
As a subgroup of $\Sp(4, \Z)$, the kernel of $\overline{\sgn} \otimes \lambda$
coincides with $\{ \pm I \}$.
\end{claim}
This claim implies the surjectivity of $h$ as follows.
Note that we have an obvious relation
$$K \subset \ker{( \text{the specialization of $\rho$ at $t=-1$} )}
	= \ker{((\overline{\sgn} \otimes \lambda) \circ \rho_0 )}.$$
Taking the images of the both ends under $\rho_0$,  we have 
$\rho_0( K ) \subset \{ \pm I \}$ by the claim.
On the other hand, 
it is easy to see that $\rho_0( \iota ) = -I$. Therefore, recalling that 
$\iota \in K$ again, we have the equality $\rho_0( \ker{\rho} ) = \{ \pm I \}$.
\par

Now, let $z$ be an arbitrary element of 
$K$. Then, we have either $\rho_0(z) = I$, or $\rho_0(z) = -I$.
In the case of  $\rho_0(z) = I$,  we have $z \in H_0$ and hence 
$z = h( 0, z )$. In the case of $\rho_0(z) = -I$,   take $x = \iota z$.
We then have $ \rho_0(x) = \rho_0(\iota) \cdot \rho_0(x) = ( -I )^2 = I$, 
and hence $x \in H_0$. Since $z = \iota x$, we have $z = h( 1, x )$.
This shows that $h$ is surjective. Therefore, we have proven
that $h$ is an isomorphism. This finishes the proof of Theorem \ref{thm-B}
\end{proof}
\par

We now prove Claim \ref{claim} to complete the proof of  Theorem \ref{thm-B}.
Most essential is  Lemma \ref{lemma} below, the proof of which is postponed until 
Appendix. It is clear that $\{ \pm I \} \subset \ker{( \overline{\sgn} \otimes \lambda )}$. 
To show the converse, 
suppose $X \in \ker{( \overline{\sgn} \otimes \lambda )}$. By the definition of 
$\overline{\sgn}$, we see that $\lambda( X )$ is equal to either $I$ or $-I$. 
On the other hand, it is easy to see that $\lambda ( -I ) = I$, and thus the representation $\lambda$ 
descends to a representation of $\PSp(4, \Z) = \Sp(4, \Z) / \{ \pm I \}$, denoted by 
$$ \bar{\lambda}: \PSp(4, \Z) \to \GL(V). $$
Here, $V$ denotes $\Lambda^2 H_1(\Sigma_2; \Z) / \omega \Z$.
Then, for the element $\overline{X} \in \PSp(4, \Z)$ corresponding to 
$X$, we have 
$\bar{\lambda}( \overline{X} ) \in \ctr{(\GL(V))}$.  Now the lemma is in order.
\begin{lemma} \label{lemma}
\begin{enumerate}
\item  $\PSp(4, \Z)$ has no nontrivial centers.
\item  As a representation of $\PSp(4, \Z)$,  $\bar{\lambda}$ is faithful.
\end{enumerate}
\end{lemma}
Then the part (2) of the lemma implies that $\overline{X}$ lies in $\ctr{(\PSp(4, \Z))}$.  
Next, by the part (1), we have $\overline{X} = I$ in $\PSp(4, \Z)$. Hence, we have $X \in \{\pm I \}$.
This completes the proof of Claim \ref{claim}. \qed
\par

\begin{remark}
By the celebrated theorem of Mess \cite{mess}, the Torelli group 
$\I_2$ is an  infinitely generated free group. Therefore,  we can see,
 by Corollary \ref{cor-B}, that $\ker{\bar{\pi}_{[ 2^3 ]}}$ is, if non-trivial, a free group.  On the other hand,  our previous results on the non-triviality 
 of $\rho_{[ 2^3 ]}$ \cite{kasahara:1, kasahara:2} could be 
considered  as showing the non-triviality of $\bar{\pi}_{[ 2^3 ]}$, and hence  of 
$\pi_{[ 2^2, 1 ]}$.
\end{remark}
\par
\begin{remark}
	It seems  quite difficult to determine whether or not the representation 
	$\pi_{[ 2^2, 1 ]}$, and therefore $\rho|_{\I_2}$, is faithful. As an example, 
	let us consider the restriction of $\pi_{[ 2^2, 1 ]}$ to $B_4$. By Proposition
	\ref{fact-2}, this restriction decomposes as 
	$\pi_{[ 2, 1^2 ]} \oplus \pi_{[ 2, 2 ]}$. Then, by a result of Long \cite{long}, 
	this 	direct sum is faithful if and only if  either one of the summands is 
	faithful. On the other hand, it is easy to see that the representation 
	$\pi_{[ 2, 2 ]}$ is {\em not faithful} since $\pi_{[ 2, 2 ]}$  can be 
	expressed as  the composition of $\pi_{[ 2, 1 ]}$ with a certain 
	homomorphism $B_4 \to B_3$ with 
	non-trivial kernel. Therefore, we can see that $\pi_{[ 2^2, 1 ]}$ is 
	faithful on $B_4$ if and only if $\pi_{[ 2, 1^2 ]}$ is faithful. 
	Now, it is well-known that $\pi_{[ 2, 1^2]}$ is equivalent to the 
	tensor product of $\sgn$ and the reduced Burau representation of $B_4$
	(see \cite[Note 5.7]{jones}). One can easily check that tensoring $\sgn$ 
	does not affect the kernel, and so we can conclude that $\pi_{[ 2^2, 1 ]}$ 
	is faithful on $B_4$ if and only if the reduced Burau representation of 
	$B_4$ is faithful. In particular, the unfaithfulness of the reduced Burau 
	representation of $B_4$ implies the unfaithfulness of $\pi_{[2^2, 1]}$ 
	and $\rho|_{\I_2}$.
	We note that the reduced Burau representation of $B_4$  is the only 
	representation among all the reduced Burau 	representations the faithfulness 
	of which remains open.
\end{remark}	
\par
\section*{Appendix---the proof of Lemma \ref{lemma}}
\par

It seems that Lemma \ref{lemma} is  well-known to experts,  but because we are unable to provide
any suitable reference in the literature, we include a proof here.
We will use some well-known properties of $\PSp(4, \F) = \Sp(4, \F)/\{ \pm I \}$,
for the details of which we refer to  \cite{grove}.
Let $p$ be a prime number. Taking the mod $p$ reduction $\Z \to \F$ for each matrix 
entry,  we obtain a homomorphism of group
$$ k_p : \Sp(4, \Z) \to \Sp(4, \F).$$
Since  $k_p( -I ) = -I$, $k_p$ induces a homomorphism 
$\bar{k}_p : \PSp(4, \Z) \to \PSp(4, \F)$.  Note that $\ker{k_p}$ consists of 
those matrices in $\Sp(4, \Z)$ for which  every  diagonal entry is equal to $1$ mod $p$ and every off-diagonal entry is  equal to $0$ mod $p$. Therefore, we can observe that 
the intersection of  $\ker{k_p}$'s, when $p$ varies among  infinitely many arbitrary
primes, consists of the single matrix $I$. It is also well-known that $\Sp(4, \F)$ is 
generated by transvections, and every transvection in $\PSp(4, \F)$ comes from 
the one in $\Sp(4, \Z)$ via $k_p$. Therefore, $k_p$ and hence $\bar{k}_p$ 
are surjective.
\par

Now we prove the first part of the lemma.  For each $X \in \Sp(4, \Z)$, we denote by 
$\overline{X}$ the corresponding element of $\PSp(4, \Z)$. Clearly, every element of 
$\PSp(4, \Z)$ has an expression of this form. Let us choose an arbitrary element 
in $\ctr{(\PSp(4, \Z))}$ and denote it by $\overline{Z}$ with $Z \in \Sp(4, \Z)$. 
Since the homomorphism $\bar{k}_p$ is surjective, $\bar{k}_p (\overline{Z})$ lies 
in $\ctr{(\PSp(4, \F))}$. We appeal to the following classical theorem.
\begin{prop_A} \label{simple_gp}
The finite group $\PSp(4, \F)$ is simple for every prime $p \geq 3$.
\end{prop_A}
In particular, for $p \geq 3$, $\PSp(4, \F)$ has no nontrivial centers.  Therefore, 
we have $\bar{k}_p ( \overline{Z} ) = I$ in $\PSp(4, \F)$. In other words, 
we have $k_p (Z) = I$, or $-I$ for $p \geq 3$. Therefore, either one of $Z$ or 
$-Z$ is contained in infinitely many $(\ker{k_p})\, \text{s'}$. Then the observation above 
implies that $Z=I$, or $-I$. In either case, we have $\overline{Z} = I$ in $\PSp(4, \Z)$.
This proves  the first part of the lemma. 
\par

Next, we proceed to (2).  Recall that the representation  $\bar{\lambda}$  of $\PSp(4, \Z)$ 
is induced by the natural action $\lambda$ of $\Sp(4, \Z)$
on the free abelian group $V = \Lambda^2 H_1(\Sigma_2; \Z)/\omega\Z$,  via the symplectic
representation $\rho_0$.
Let $V_p$ denote $V \otimes \F$. We consider the representation
$\lambda \otimes 1_{\F}: \Sp(4, \Z) \to \GL(V_p)$. Here, $1_{\F}$ denotes the identity 
on $\F$. Clearly, we have $\ker{k_p} \subset \ker{(\lambda \otimes 1_{\F})}$, and 
thus the representation $\lambda \otimes 1_{\F}$ descends to that of $\Sp(4, \F)$, 
denoted by 
$$ \lambda_p : \Sp(4, \F) \to \GL(V_p).$$
Furthermore, recalling that $\lambda(-I)=I$ in $\GL(V)$,  we see that 
$\lambda_p$ descends to a representation  of $\PSp(4, \F)$ denoted by
$\bar{\lambda}_p$. It is easy to check that $\bar{\lambda}_p$ is nontrivial for arbitrary 
prime $p$, for instance by computing the image  under 
${\lambda}_p \circ k_p \circ \rho_0$ of the Dehn twist along any non-separating simple
closed curve.  Then, by 
Proposition A again, $\bar{\lambda}_p$ is a faithful representation 
of $\PSp(4, \F)$ for $p \geq 3$. Therefore, for arbitrary $X \in \ker{\lambda}$, we have $k_p (X) = I$, 
or $-I$. Now the same argument as in (1) implies $\overline{X} = I$ in $\PSp(4, \Z)$.
This proves (2), and hence the lemma. \qed
\par

We remark that the same proof as above works for 
$\PSp(2g, \Z)$ with general $g \geq 2$,  and its nontrivial linear representation
defined naturally on an arbitrary free abelian subquotient of the tensor 
product of copies of the fundamental representation.
\par
\providecommand{\bysame}{\leavevmode\hbox to3em{\hrulefill}\thinspace}
\providecommand{\MR}{\relax\ifhmode\unskip\space\fi MR }

\providecommand{\MRhref}[2]{%
  \href{http://www.ams.org/mathscinet-getitem?mr=#1}{#2}
}
\providecommand{\href}[2]{#2}


\begin{thebibliography}{99}

\bibitem{AMU}
J.~E. Andersen, G.~Masbaum, and K.~Ueno, \emph{Topological quantum field theory
  and the {N}ielsen-{T}hurston classification of {$M(0,4)$}}, Math. Proc.
  Cambridge Philos. Soc. \textbf{141} (2006), 477--488.

\bibitem{bigelow}
S.~Bigelow, \emph{Braid groups are linear}, J. Amer. Math. Soc. \textbf{14}
  (2001), no.~2, 471--486.

\bibitem{bigelow-budney}
S.~Bigelow and R.~Budney, \emph{The mapping class group of a genus two surface
  is linear}, Algebr. Geom. Topol. \textbf{1} (2001), 699--708.

\bibitem{birman-hilden}
J.~S. Birman and H.~Hilden, \emph{Isotopies of homeomorphisms of {R}iemann
  surfaces and a theorem about {A}rtin's braid group}, Ann. of Math.
  \textbf{97} (1973), 424--439.

\bibitem{gillette-buskirk}
R.~Gillette and J.~van Buskirk, \emph{The word problem and consequences for the
  braid groups and mapping class groups of the 2-sphere}, Trans. Amer. Math.
  Soc. \textbf{131} (1968), 277--296.

\bibitem{grove}
L.~C. Grove, \emph{Classical groups and geometric algebra}, Grad. Stud. in
  Math., vol.~39, American Mathematical Society, Providence, RI, 2002.

\bibitem{jones}
V.~F.~R. Jones, \emph{Hecke algebra representations of braid groups and link
  polynomials}, Ann. of Math. \textbf{126} (1987), 335--388.

\bibitem{kasahara:1}
Y.~Kasahara, \emph{An expansion of the {J}ones representation of genus $2$ and
  the {T}orelli group}, Algebr. Geom. Topol. \textbf{1} (2001), 39--55.

\bibitem{kasahara:2}
\bysame, \emph{An expansion of the {J}ones representation of genus $2$ and the
  {T}orelli group {II}}, J. Knot Theory Ramifications \textbf{13} (2004),
  297--306.

\bibitem{kasahara:3}
\bysame, \emph{The {J}ones representation of genus $1$}, to appear in: 
  Intelligence of Low
  Dimensional Topology 2006, Series on Knot and Everything, vol.~40, World Sci.
  Publ., River Edge, NJ, 2007.

\bibitem{korkmaz}
M.~Korkmaz, \emph{On the linearity of certain mapping class groups}, Turkish J.
  Math. \textbf{24} (2000), no.~4, 367--371.

\bibitem{krammer}
D.~Krammer, \emph{Braid groups are linear}, Ann. of Math. (2) \textbf{155}
  (2002), no.~1, 131--156.

\bibitem{long}
D.~Long, \emph{A note on the normal subgroups of mapping class groups}, Math.
  Proc. Camb. Philos. Soc. \textbf{99} (1986), 79--87.

\bibitem{magnus-peluso}
W.~Magnus and A.~Peluso, \emph{On a theorem of {V}. {I}. {Arnol'd}}, Comm. Pure
  Appl. Math. \textbf{22} (1969), 683--692.

\bibitem{mathas}
A.~Mathas, \emph{{I}wahori-{H}ecke algebras and {S}chur algebras of the
  symmetric group}, University Lecture Series. 15., American Mathematical
  Society, 1999.

\bibitem{mess}
G.~Mess, \emph{The {T}orelli groups for genus 2 and 3 surfaces}, Topology
  \textbf{31} (1992), 775--790.

\end{thebibliography}
\end{document}